\documentclass[12pt,reqno]{amsart}
\usepackage{amssymb}
\usepackage{amsmath}

\usepackage{amscd}

\newcommand{\RNum}[1]{\uppercase\expandafter{\romannumeral #1\relax}}

\usepackage[T2A]{fontenc}
\usepackage[utf8]{inputenc}
\usepackage[english]{babel}
\input{int.def}

\usepackage{tikz-cd}
\usetikzlibrary{cd}
\usepackage{dirtytalk}
\usepackage{hyperref}
\usepackage{xcolor}
\usepackage{centernot}

\usepackage{enumitem}

\usepackage{pgfplots}
\usepackage{multicol}

\pgfplotsset{compat=1.17}

\numberwithin{equation}{section}

\usepackage[mathcal]{euscript}

\usepackage{titlesec}
\titleformat{\section}[runin]{\bfseries}{\thesection.}{3pt}{}[.]

\myitemmargin 
\usepackage{geometry}
\newgeometry{vmargin={25mm}, hmargin={22mm,22mm}, footskip=10mm}   

\begin{document}

\title[Detecting adversarial attacks on random samples]%
{Detecting adversarial attacks on random samples}


\author{Gleb Smirnov}
\address{School of Mathematics, 
University of Geneva, Rue du Conseil-G{\'e}n{\'e}ral, 7, 1205, Geneva, Switzerland}

\email{gleb.smirnov@unige.ch}




\begin{abstract}
This paper studies the problem of detecting adversarial perturbations in a sequence of observations. Given a data sample $X_1, \ldots, X_n$ drawn from a standard normal distribution, an adversary, after observing the sample, can perturb each observation by a fixed magnitude or leave it unchanged. We explore the relationship between the perturbation magnitude, the sparsity of the perturbation, and the detectability of the adversary's actions, establishing precise thresholds for when detection becomes impossible.
\end{abstract}

\maketitle

\setcounter{section}{0}

\section{Introduction}\label{intro}

We consider the problem of testing a simple hypothesis against a complex alternative. Let $\boldsymbol{X} = (X_1, \ldots, X_n)$ be a data sample drawn from the standard normal distribution, $N(0,1)$. An adversary, after observing $\boldsymbol{X}$, can either attack or leave the data unchanged based on a signal from a third party. If an attack occurs, the adversary adds a vector $\boldsymbol{\theta} = (\theta_1, \ldots, \theta_n)$ from a subset $\Theta \subset \mathbb{R}^n$ to $\boldsymbol{X}$. The choice of $\boldsymbol{\theta}$ may depend on the original data $\boldsymbol{X}$, unlike the typical hypothesis testing scenarios. The perturbed observation is denoted by $\boldsymbol{X}'$. We define the hypotheses as:
\[
H_0: \boldsymbol{X}' = \boldsymbol{X} \quad \text{(no attack)}
\]
\[
H_1: \boldsymbol{X}' = \boldsymbol{X} + \boldsymbol{\theta}(\boldsymbol{X}) \quad \text{(attack)}.
\]
The classifier must decide, based on $\boldsymbol{X}'$ and the set of possible perturbations $\Theta$, whether an attack occurred ($H_1$) or not ($H_0$). The test accepts $H_0$ if $\boldsymbol{X}'$ lies within a Borel subset $\Omega \subset \mathbb{R}^n$:
\[
\text{Accept } H_0 \text{ iff } \boldsymbol{X}' \in \Omega.
\]
The objective is to minimize the false negative rate,
\[
P(\text{accept } H_0 \mid H_1),
\]
while keeping the false alarm rate below a threshold $\alpha$ (e.g., 0.05):
\[
P(\text{accept } H_1 \mid H_0) \leq \alpha.
\]
The adversary has full access to the classifier's test 
(white-box attack), in principle allowing them to craft a highly effective perturbation. However, in the scenario considered here, this access is essentially useless — the adversary gains little to no advantage from knowing the test. The classifier knows the set $\Theta$ but not the adversary's strategy for choosing $\boldsymbol{\theta}$. However, in our example, revealing the adversary's strategy does not benefit the classifier. 
\smallskip%

For the set of potential attacks $\Theta \subset \mathbb{R}^n$ (again, known to both the classifier and the adversary), some regularity conditions are necessary to ensure that related sets are Borel measurable. In this paper, we consider two examples of $\Theta$: an $(n-1)$-dimensional sphere in $\mathbb{R}^n$ and a finite set of points in $\mathbb{R}^n$. Both cases satisfy the required regularity.
\smallskip%

The problem can also be framed geometrically. Given a Borel set $\Omega \subset \mathbb{R}^n$, we translate $\Omega$ by $-\boldsymbol{\theta}$ for each $\boldsymbol{\theta} \in \Theta$ and take the union of these translations:
\[
\Omega - \Theta = \left\{ \boldsymbol{x} \in \mathbb{R}^n \mid \boldsymbol{x} + \boldsymbol{\theta} \in \Omega \quad \text{for some } \boldsymbol{\theta} \in \Theta \right\}.
\]
Here, $-\Theta$ denotes the set obtained by applying the antipodal map to $\Theta$. Thus, $\Omega - \Theta$ is the Minkowski sum of $\Omega$ and $-\Theta$, hence the notation.
\smallskip%

Let $\boldsymbol{X} = (X_1, \ldots, X_n)$ be a random vector in $\mathbb{R}^n$ with iid components, each $N(0,1)$, and let $\gamma$ be the corresponding Gaussian measure on $\mathbb{R}^n$. The task is to bound the Gaussian volume of $\Omega - \Theta$ in terms of that of $\Omega$. Specifically, we have:
\[
\gamma(\Omega) = 1 - P\left( \text{accept } H_1 \mid H_0 \right), \quad 
\gamma(\Omega - \Theta) = P\left( \text{accept } H_0 \mid H_1 \right).
\]
We now describe the specific scenario considered in this paper. Fix a number $a > 0$, for now constant, independent of $n$. For each $i = 1, \ldots, n$, let
\[
\theta_i \in \{ \pm a, 0 \},
\]
meaning the adversary can perturb each coordinate $X_i$ by $\pm a$ or leave it unchanged. The strength of the adversarial perturbation is measured using the sparsity ratio of $\boldsymbol{\theta}$, defined as:
\[
\text{\normalfont{SR}}(\boldsymbol{\theta}) = \frac{\#\{ i : \theta_i = 0 \}}{n}.
\]
Intuitively, attacks with a very high sparsity ratio should be nearly impossible to detect. To ensure a meaningful problem setup, we should impose an upper bound on the perturbation strength. Conversely, when many data points are altered, the attack becomes easier for the classifier to detect. We seek to determine the threshold for the sparsity ratio above which the adversarial attack becomes undetectable by the classifier.
\smallskip%

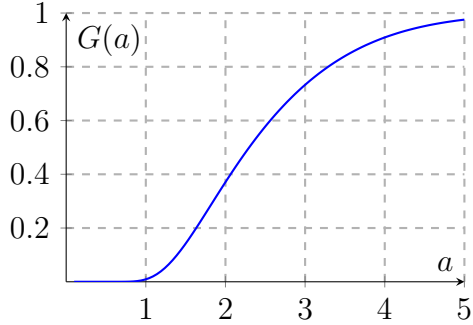
\begin{center}
\begin{tikzpicture}
    \begin{axis}[
        width=0.4\textwidth, 
        height=0.3\textwidth,
        xlabel={$a$},
        ylabel={$G(a)$},
        ymin=0, ymax=1,
        xmin=0, xmax=5,
        samples=100,
        domain=0.1:5,
        grid=both,
        grid style={dashed, gray!30},
        major grid style={thick, black!30},
        axis lines=middle,
        legend pos=south east
    ]
    \addplot[
        blue,
        thick
    ]{
        (4/pi) * (
            (-1)^0 / (1 + 2*0) * exp(-((1 + 2*0)^2 * pi^2) / (2*x^2)) +
            (-1)^1 / (1 + 2*1) * exp(-((1 + 2*1)^2 * pi^2) / (2*x^2)) +
            (-1)^2 / (1 + 2*2) * exp(-((1 + 2*2)^2 * pi^2) / (2*x^2)) 
        )
    };
    \end{axis}
\end{tikzpicture}
\end{center}

To state the main result, we define the following function:
\begin{equation}\label{g}
g(x) = \dfrac{1}{\sqrt{2 \pi}} \sum_{-\infty}^{+\infty} (-1)^k 
e^{-(x + k a)^2/2},
\end{equation}
and set:
\begin{equation}\label{G}
G(a) = \int_{-a/2}^{a/2} g(x)\, dx.
\end{equation}
In Lemma \ref{poisson} above, this 
integral is calculated as:
\begin{equation}\label{G-int}
G(a) = \dfrac{4}{\pi} \sum_{k \geq 0} 
\frac{(-1)^k}{1+2 k} \exp\left( -\frac{(1+2k)^2\pi^2}{2a^2}\right).
\end{equation}
For $a > 0$, $G(a)$ is smooth and strictly increasing, with 
\[
\lim_{a \to  0} G(a) = 0\quad\text{and}\quad
\lim_{a \to  \infty} G(a) = 1.
\]
\begin{proposition}[Sparse attacks]\label{t_main}
Let $\varepsilon > 0$ be independent of $n$ and small enough so that 
\[
0 < G(a) - \varepsilon < G(a) + \varepsilon < 1,
\]
but otherwise arbitrary. 
\begin{enumerate}
\item Suppose that the adversary can only perform those attacks which satisfy
\[
\operatorname{SR}(\boldsymbol{\theta}) < 
G(a) - \varepsilon.
\]
Then for every $\alpha > 0$ there exists a test\footnote{Strictly speaking, a sequence of tests indexed by $n$.} such that for all large enough $n$,
\[
\pp\left( \text{\normalfont{accept $H_1$}}\,|\, H_0 \right) \leq \alpha\quad 
\text{\normalfont{and}}\quad 
\pp\left( \text{\normalfont{accept $H_0$}}\,|\, H_1 \right) \leq \alpha.
\]
\item Conversely, based on the received observations, the adversary can design an attack $\boldsymbol{\theta}$ with only slightly higher sparsity, 
\[
\operatorname{SR}(\boldsymbol{\theta}) < 
G(a) + \varepsilon,
\]
such that regardless of classifier's test,
\begin{equation}\label{errors_sum}
P\left( \text{\normalfont{accept $H_0$}}\,|\, H_1 \right) + 
P\left( \text{\normalfont{accept $H_1$}}\,|\, H_0 \right) > 1 - O(e^{-n \varepsilon^2}).
\end{equation}
\end{enumerate}
\end{proposition}
The proof of (1) proceeds as follows: First, in \S\,\ref{coin_testing}, a related but simpler testing problem is considered. A biased coin with probability $p \neq 1/2$ of $1$ (and $(-1)$ otherwise) is tossed $n$ times, generating an iid sequence $Z_1, \ldots, Z_n$. After observing the sequence, the adversary flips a fraction of the outcomes. The classifier, knowing $p$, must decide if changes have been made. Part (1) is then proved in \S\,\ref{g_testing} by reducing the original problem to this coin testing scenario. The proof is straightforward, relying only on computing the integral \eqref{G} and applying Hoeffding's inequality.
\smallskip%

The proof of (2), found in \S\,\ref{generator}, involves constructing an explicit adversarial strategy. A key approach for the adversary is to ignore the classifier's test, even though it is known, and instead craft $\boldsymbol{\theta} = \boldsymbol{\theta}(\boldsymbol{X})$ such that $\boldsymbol{X}'$ remains an iid sequence of Gaussians. If this is achieved, the adversary can challenge the classifier with $\boldsymbol{X}'$, making it appear as if no attack occurred. This method is first demonstrated in a simpler setting of fixed $\ell_2$-norm attacks, where the adversary chooses any $\boldsymbol{\theta} \in \mathbb{R}^n$ with $\| \boldsymbol{\theta} \|_2 = R$. Here, $\Theta \subset \mathbb{R}^n$ is a hypersphere of radius $R$. This toy example is detailed in \S\,\ref{ell2}. Designing optimal attacks for the original problem is more complex than in the $\ell_2$-norm case, and the dependence of $\boldsymbol{\theta}$ on $\boldsymbol{X}$ involves Jacobi theta functions.
\smallskip%

Proposition \ref{t_main} provides quantitative results, including an estimate of the sample size required for effective testing, when one is possible. The big-O term in \eqref{errors_sum} comes from Hoeffding's bound. 
\smallskip%

It is instructive to consider the case where the adversary is not allowed to leave any coordinate $X_i$ unchanged, meaning $X'_i = X_i \pm a$. Here, $\Theta$ corresponds to the set of vertices of a hypercube with side length $2a$. Proposition \ref{t_main} suggests that, as long as $a$ is independent of $n$, no consistent testing is possible unless $a = 0$. We determine the rate at which $a$ must decay for the classifier to become unable to detect an attack. 
\begin{proposition}[Hypercube]\label{t_cube}
Assume 
\[
a = c/\sqrt{\ln n}
\]
for some constant $c$ independent of $n$. Suppose that for each $i$, 
\[
\theta_i = \left\{-a, a\right\}.
\]
\begin{enumerate}
\item If $c > \pi$, then for every $\alpha > 0$ there exists a test such that 
for all large enough $n$, 
\[
P\left( \text{\normalfont{accept $H_1$}}\,|\, H_0 \right) \leq \alpha\quad 
\text{\normalfont{and}}\quad 
P\left( \text{\normalfont{accept $H_0$}}\,|\, H_1 \right) \leq \alpha.
\]
\item Conversely, if $c < \pi/\sqrt{2}$, then, 
based on the received observations, the adversary can design an attack 
such that regardless of classifier's test,
\[
P\left( \text{\normalfont{accept $H_0$}}\,|\, H_1 \right) + 
P\left( \text{\normalfont{accept $H_1$}}\,|\, H_0 \right) > 1 - O\left( n^{ 1 - \frac{\pi^2}{2c^2}} \right).
\]
\end{enumerate}
\end{proposition}
The proof follows from Proposition \ref{t_main} by carefully analyzing the regime where $a \sim 1/\sqrt{\ln n}$.
\smallskip%

It is instructive to restate the result of Proposition \ref{t_cube} geometrically, in terms of volume bounds on the Minkowski sum of an arbitrary domain $\Omega$ and the hypercube $\Theta$. Since $\Theta$ is centrally symmetric, $\Omega + \Theta$ coincides with $\Omega - \Theta$. We obtain the following:
\begin{corollary}[ = Proposition \ref{t_cube}]
Assume 
\[
a = c/\sqrt{\ln n}
\]
for some constant $c$ independent of $n$. Let $\Theta$ be the set of vertices of a hypercube of side length $2a$, i.e.,
\[
\Theta = \left\{ (\theta_1, \ldots, \theta_n) \in \mathbb{R}^n \mid \theta_i = \pm a \text{ for } i = 1, \ldots, n \right\}, \quad \left|\Theta\right| = 2^n.
\]
\begin{enumerate}
\item If $c > \pi$, then for every $\alpha > 0$ and all large enough $n$, 
there exists a Borel set $\Omega$ such that 
\[
\gamma(\overline{\Omega}) \leq \alpha \quad 
\gamma(\Omega + \Theta) \leq \alpha.
\]
\item Conversely, if $c < \pi/\sqrt{2}$, then for each Borel set $\Omega$, 
\[
\gamma(\overline{\Omega}) + \gamma(\Omega + \Theta) > 1 - O\left( n^{ 1 - \frac{\pi^2}{2c^2}} \right).
\]
\end{enumerate}
\end{corollary}
We also state a related corollary, which follows from \S\,\ref{ell2}, providing a bound on the Gaussian volume of the Minkowski sum of $\Omega$ and a hypersphere.

\begin{corollary}[ = Proposition \ref{t_sphere}]
Let $S_{R}$ be the sphere of radius $R$ in $\mathbb{R}^n$, i.e.,
\[
S_{R} = \left\{ (\theta_1, \ldots, \theta_n) \in \mathbb{R}^n \mid \sum_{i = 1}^n \theta_i^2 = R^2 \right\},
\]
and let $B_{R/2}$ denote the ball of radius $R/2$ in $\mathbb{R}^n$, i.e.,
\[
B_{R/2} = \left\{ (x_1, \ldots, x_n) \in \mathbb{R}^n \mid \sum_{i = 1}^n x_i^2 \leq R^2/4 \right\}.
\]
For each Borel set $\Omega$, we have
\[
\gamma(\overline{\Omega}) + \gamma(\Omega + S_{R}) > 1 - 
\gamma(B_{R/2}).
\]
\end{corollary}
It is instructive to compare the above testing problem with the setting where the adversary must choose a perturbation \emph{before} observing the data. In this case, the vector $\boldsymbol{\theta}$ is independent of $\boldsymbol{X}$, making the problem equivalent to detecting a signal ($\boldsymbol{\theta}$) in Gaussian noise ($\boldsymbol{X}$), a well-studied problem in statistics; see, e.g., \cite{Burn, Ing-2, 
AC, AC-2, AC-3, quatro, Ing-Tsyb, Ing-Tsyb-2, Donoho-J, Donoho-J-2, Donoho-J-3, Baraud, Cai, Carp, Mukh-Mukh} and references therein. To emphasize the difference from the adversarial testing problem addressed here, consider a specific scenario. Let $\boldsymbol{X} = (X_1, \ldots, X_n)$ be iid random variables, each distributed as $N(0,1)$. Without observing the realizations of $X_i$, the adversary selects a perturbation $\boldsymbol{\theta}$ such that $\theta_i \in \{-a, a\}$ for all $i$. The classifier must then decide between $H_0$ and $H_1$ after receiving a potentially contaminated sample. Although the set of possible signals $\boldsymbol{\theta}$ forms a hypercube, similar to Proposition \ref{t_cube}, the rate of decay of $a$ sufficient for consistent testing is much higher than $1/\sqrt{\ln{n}}$. To show this, observe that, since $X_i$ and $\theta_i$ are independent, $(X_i + \theta_i)^2$ and $(X_i - \theta_i)^2$ are identically distributed. Under $H_1$, the squared norm of $(X'_1, \ldots, X'_n)$ is as follows:
\[
\|\boldsymbol{X}'\|^2 = \sum_{i} (X_i + \theta_i)^2 \overset{d}{=} \sum_{i} (X_i + a)^2.
\]
We compute:
\[
\mathbb{E}((X+a)^2) = 1 + a^2, \quad 
\operatorname{Var}((X+a)^2) = 2(1 + 2a^2).
\]
Using Chebyshev's inequality, we get:
\[
P\left( | n^{-1/2} \sum_i (X_i + a)^2 - n^{1/2} (1 + a^2) | \geq \varepsilon \right) \leq 
2 (1+2a^2) \varepsilon^{-2} < 6 \varepsilon^{-2} \quad \text{for } a < 1.
\]
Under $H_1$, if $a$ decays slower than $n^{-1/4}$, the normalized squared norm $\|\boldsymbol{X}'\|^2$ concentrates around $\sqrt{n}(1+a^2)$, significantly exceeding $\sqrt{n}$. Conversely, under $H_0$, we get:
\[
P\left( | n^{-1/2} \sum_i X_i^2 - n^{1/2} | \geq \varepsilon \right) \leq 2 \varepsilon^{-2}.
\]
This example illustrates that data-dependent adversarial perturbations are significantly harder to detect than fixed or randomly chosen perturbations.
\smallskip%

While there is a large body of literature on random perturbations, adversarial settings are less documented.
The work \cite{Versh} addresses the adversarial testing problem considered here and provides a complete solution for a class of so-called \emph{highly symmetric} sets $\Theta$. Although this class is broad and natural, it does not cover the examples considered in the present paper. In addition to solving the problem for highly symmetric sets, \cite{Versh} introduces the concept of \emph{focused Gaussian width}, which is well-suited for this problem. The detection tests in their work can also be applied to non-Gaussian data, unlike the test in Proposition \ref{t_main}, which is designed as a normality test.

\section{$\ell_2$-attacks}\label{ell2}
To illustrate the approach used in \S\,\ref{generator}, 
we consider a simpler problem. Fix a radius $R > 0$, known to the adversary and the classifier. The adversary observes a sample $(X_1, \ldots, X_n)$, $X_i \sim N(0,1)$, and can select any data-dependent perturbation $(\theta_1, \ldots, \theta_n)$ such that 
\begin{equation}\label{theta_R}
\sum_i \theta_i^2 = R^2.
\end{equation}
\begin{proposition}[Sphere]\label{t_sphere}
Based on the received observations, the adversary 
can design an attack $\boldsymbol{\theta}$ satisfying \eqref{theta_R} such that, regardless of the 
classifier's test,
\[
P\left( \text{\normalfont{accept $H_0$}}\,|\, H_1 \right) + 
P\left( \text{\normalfont{accept $H_1$}}\,|\, H_0 \right) > 1 - P(B_{R/2}),
\]
where $B_{R/2}$ is the ball of radius $R/2$, and $P(B_{R/2})$ is the probability that $(X_1, \ldots, X_n)$ falls inside $B_{R/2}$.
\end{proposition}
\begin{proof}
To begin, we define a random vector $\boldsymbol{Y}$ as follows. If $\boldsymbol{X} = (X_1, \ldots, X_n)$ falls outside the ball $B_{R/2}$, where $B_{R/2}$ is the ball of radius $R/2$ centered at the origin, we proceed by constructing a sphere $C_R$ of radius $R$ centered at $\boldsymbol{X}$. Let $S_{\|\boldsymbol{X}\|}$ denote the sphere of radius $\|\boldsymbol{X}\|$ centered at the origin. The intersection of the spheres $S_{\|\boldsymbol{X}\|}$ and $C_R$ forms another sphere, but of dimension $n-2$. We endow this resulting $(n-2)$-dimensional sphere with the uniform measure and choose $\boldsymbol{Y}$ as a random point on it. On the other hand, if $\boldsymbol{X}$ lies inside $B_{R/2}$, we set $\boldsymbol{Y} = \boldsymbol{X}$.
\smallskip%

We claim that $\boldsymbol{Y}$ is again a Gaussian vector with iid components, each $N(0,1)$. By construction, the norm of $\boldsymbol{Y}$ equals that of $\boldsymbol{X}$, and the direction of $\boldsymbol{Y}$ is uniformly distributed, as it is for $\boldsymbol{X}$. Additionally, the direction of $\boldsymbol{Y}$ is independent of its norm. These properties characterize the Gaussian distribution in $\mathbb{R}^n$, and the claim follows.
\smallskip%

The adversary's strategy for constructing \( \boldsymbol{X}' \) is as follows: Given \( \boldsymbol{X} \), the adversary evaluates \( \boldsymbol{Y} \) as described above. If \( \boldsymbol{X} \) lies outside \( B_{R/2} \), then set \( \boldsymbol{X}' = \boldsymbol{Y} \). If \( \boldsymbol{X} \) is inside \( B_{R/2} \), choose \( \boldsymbol{X}' \) to be any point at a distance \( R \) from \( \boldsymbol{X} \). Under this strategy, we derive the following lower bound on the false-negative rate:
\begin{multline*}
P\left( \text{\normalfont{accept }} H_0 \,|\, H_1 \right) = P(\boldsymbol{X}' \in \Omega) \geq P(\boldsymbol{X} \notin B_{R/2}, \boldsymbol{Y} \in \Omega) \geq \\
P(\boldsymbol{X} \notin B_{R/2}) + P(\boldsymbol{Y} \in \Omega) - 1 = P(\boldsymbol{X} \notin B_{R/2}) + P(\boldsymbol{X} \in \Omega) - 1.
\end{multline*}
By definition, \( P\left( \text{\normalfont{accept }} H_1 \,|\, H_0 \right) = 1 - P(\boldsymbol{X} \in \Omega) \), and the proposition follows. \qed
\end{proof}

\section{Coin testing}\label{coin_testing} 
To establish a key lemma for \S\,\ref{g_testing}, we consider adversarial attacks on Bernoulli-distributed data. Let \( Z \) be a random variable with \( P(Z = 1) = p > 1/2 \) and \( P(Z = -1) = 1 - p \). Both the adversary and the classifier know \( p \) and agree on a number \( \varepsilon > 0 \), which will measure 
the attack strength. The adversary observes independent trials \( (Z_1, \ldots, Z_n) \) from \( Z \) and selects a subset of \( Z_i \), of fraction less than \( (2p - 1) - \varepsilon \), to leave intact while flipping the signs of the rest. The classifier receives the possibly contaminated sequence and must choose between \( H_0 \) (no flips) and \( H_1 \) (flipped). We show that if the fraction of unaffected points is less than \( 2p - 1 - \varepsilon \), reliable testing is always possible.
\begin{lemma}\label{coin_lemma}
Let \( \varepsilon > 0 \) be independent of \( n \) and small enough that \( (2p - 1) - \varepsilon > 0 \). Suppose that the adversary flips all but a fraction less than \( (2p - 1) - \varepsilon \) of \( Z_1, \ldots, Z_n \). Then for every \( \alpha > 0 \), there exists a test such that for all large enough $n$,
\[
P(\text{\normalfont{accept }} H_1 \,|\, H_0) \leq \alpha, \quad \text{\normalfont{and}} \quad P(\text{\normalfont{accept }} H_0 \,|\, H_1) \leq \alpha.
\]
\end{lemma}
\begin{proof}
Consider \( A_n(Z_1, \ldots, Z_n) = n^{-1} \sum_i Z_i \). Hoeffding's inequality gives, for each \( \lambda > 0 \),
\begin{equation}\label{noga}
P(A_n - \mathbb{E}(A_n) \leq -\lambda/\sqrt{n}) \leq e^{-\lambda^2/2}, \quad \mathbb{E}(A_n) = 2p - 1.
\end{equation}
Choose \( \lambda \) such that \( e^{-\lambda^2/2} < \alpha \). Let \( Z'_1, \ldots, Z'_n \) be the observations received by the classifier, and define the test:
\[
\Omega = \left\{ \sqrt{n}(A_n(Z'_1, \ldots, Z'_n) - (2p - 1)) > -\lambda \right\}.
\]
Under \( H_0 \) (no attack), the false alarm probability is \( \leq \alpha \). Under \( H_1 \) (attack), we obtain:
\begin{multline*}
n^{-1} \sum_i Z'_i = 
n^{-1} \left( \sum_{\text{flipped}} Z'_i + \sum_{\text{not flipped}} Z'_i  \right) = 
n^{-1} \left( -\sum_{\text{flipped}} Z_i + \sum_{\text{not flipped}} Z_i  \right) = \\
= -n^{-1} \sum_i Z_i + 2n^{-1} \sum_{\text{not flipped}} Z_i < 
-n^{-1} \sum_i Z_i + 2(2p - 1 - \varepsilon),
\end{multline*}
where \say{flipped} refers to points changed by the adversary. Thus if $Z_1, \ldots, Z_n$ and $n$ are such that 
\[
n^{-1} \sum_i Z_i - (2p - 1) > -\lambda/\sqrt{n},\quad n > \lambda^2/\varepsilon^2,
\]
we get:
\[
n^{-1} \sum_i Z'_i - (2p - 1) < 
-n^{-1} \sum_i Z_i + (2p - 1) - 2\varepsilon < \lambda/\sqrt{n} - 2\varepsilon < \lambda/\sqrt{n}.
\]
Consequently, under $H_1$, the events
\[
\sqrt{n}(A_n(Z'_1, \ldots, Z'_n) - (2p - 1)) > -\lambda, \qquad 
\sqrt{n}(A_n(Z_1, \ldots, Z_n) - (2p - 1)) > -\lambda
\]
are mutually exclusive, and the lemma follows. \qed
\end{proof}
We now analyze the case where the adversary is forced to flip the signs of all observations \( Z_i \), while letting \( p \) approach \( 1/2 \), but not too quickly.
\begin{lemma}\label{coin_lemma_n}
Assume \( \sqrt{n}(2p - 1) \to \infty \) as \( n \to \infty \), and the adversary is no longer allowed to leave any \( Z_i \) intact. Then for every \( \alpha > 0 \), there exists a test such that for all large enough $n$,
\[
P(\text{\normalfont{accept }} H_1 \,|\, H_0) \leq \alpha, \quad \text{\normalfont{and}} \quad P(\text{\normalfont{accept }} H_0 \,|\, H_1) \leq \alpha.
\]
\end{lemma}
\begin{proof}
We define a simpler test as:
\[
\Omega = \left\{ A_n(Z'_1, \ldots, Z'_n) > 0 \right\}.
\]
Choose \( n \) large enough that \( \sqrt{n}(2p - 1) > \lambda \). From \eqref{noga}, we get:
\[
A_n(Z_1, \ldots, Z_n) > A_n(Z_1, \ldots, Z_n) - \left[ (2p - 1) - \lambda/\sqrt{n} \right] > 0
\]
with probability at least \( 1 - e^{-\lambda^2/2} \). 
Under \( H_1 \), we obtain:
\[
\sum_i Z'_i = -\sum_i Z_i.
\]
Consequently, under $H_1$, the events
\[
A_n(Z'_1, \ldots, Z'_n) > 0, \qquad 
A_n(Z_1, \ldots, Z_n) > 0
\]
are mutually exclusive, and the lemma follows. \qed
\end{proof}

\section{Proof of (1) of Propositions \ref{t_main} and \ref{t_cube}}\label{g_testing}
Let \( X_1, \ldots, X_n \) be iid random variables, each \( N(0,1) \), representing the original observations. Partition the real line \( \mathbb{R} \) into a countable union of bins \( U_k \), defined as:
\[
U_k = \left\{ x \in \mathbb{R} : k a - \frac{a}{2} < x < k a + \frac{a}{2} \right\}, \quad k \in \mathbb{Z}.
\]
Whatever the values of $X_i$, chances are all of them lie within the union of 
$U_k$. Define the random variable $Z_i$ associated with $X_i$ as:
\[
Z_i = 
\begin{cases} 
1 & \text{if } X_i \in U_k \text{ for some even } k, \\
-1 & \text{if } X_i \in U_k \text{ for some odd } k.
\end{cases}
\]
Let $p_k$ denote the probability measure of $U_k$, i.e.,
\[
p_k = (2 \pi)^{-1/2} \int_{k a -a/2}^{ka + a/2} e^{-x^2/2} \, dx.
\]
Each \( Z_i \) is a Bernoulli random variable with
\begin{equation}\label{Z_bern}
P(Z_i = 1) - P(Z_i = -1) = \sum_{k} (-1)^k p_k.
\end{equation}
Using the function $G$ defined in \eqref{G}, we get:
\[
\sum_{k} (-1)^k p_k = (2\pi)^{-1/2}\sum_{k} \int_{k a -a/2}^{ka + a/2} e^{-x^2/2} \, dx = 
(2\pi)^{-1/2} \sum_{k} \int_{-a/2}^{a/2} e^{-(x + k a)^2/2} \, dx = G(a).
\]
In the event of an attack, the adversary replaces the original observations with contaminated observations \( X'_1, \ldots, X'_n \), where \( X'_i = X_i + \theta_i \), and
\[
\theta_i \in \{-1, 0, 1\}, \quad 
\frac{\#\{ i : \theta_i = 0 \}}{n} < G(a) - \varepsilon.
\]
Similar to $Z_i$, define $Z'_{i}$ as:
\[
Z'_i = 
\begin{cases} 
1 & \text{if } X'_i \in U_k \text{ for some even } k, \\
-1 & \text{if } X'_i \in U_k \text{ for some odd } k.
\end{cases}
\]
Observe that
\[
Z'_i = 
\begin{cases} 
Z_i & \text{if } \theta_i = 0, \\
-Z_i & \text{otherwise},
\end{cases}
\]
Thus, under \( H_0 \) (no attack), we have \( Z'_i = Z_i \), and the classifier observes a Bernoulli sequence characterized by \eqref{Z_bern}. Under \( H_1 \) (attack), the observed sequence \( Z'_i \) results from flipping the signs of \( Z_i \), with fewer than \( G(a) - \varepsilon \) of the observations not flipped.
\smallskip%

As shown in Lemma \ref{coin_lemma}, it is possible to consistently distinguish between these two cases, regardless of how small \( \varepsilon \) is. This completes the proof of (1) of Proposition \ref{t_main}.
\smallskip%

Set \( a = c/\sqrt{\ln{n}} \). To establish Proposition \ref{t_cube}, it suffices to show that if \( c > \pi \), then \( \sqrt{n} G(a) \to \infty \) as \( n \to \infty \). The proof then follows from Lemma \ref{coin_lemma_n}. 
\smallskip%

Substituting $a = c/\sqrt{\ln{n}}$ into \eqref{G-int}, we get:
\begin{multline*}
G(a) = \dfrac{4}{\pi} \sum_{k \geq 0} 
\frac{(-1)^k}{1+2 k} 
e^{-(1+2k)^2 \cdot \pi^2/2a^2} > 
\dfrac{4}{\pi} \left[
e^{-\pi^2/2a^2} - \frac{1}{3} e^{-9 \cdot \pi^2/2a^2} + \ldots \right] > 
\dfrac{4}{\pi} e^{-\pi^2/2a^2} \left[
1 - \frac{1}{3} e^{-4\pi^2/a^2}
\right] =  \\
= \dfrac{4}{\pi} n^{-\pi^2/2c^2}\left[
1 - \frac{n^{-4\pi^2/c^2}}{3} \right] > 
n^{-\pi^2/2c^2}\quad \text{for $n$ large enough.}
\end{multline*}
This completes the proof of (2) of Proposition \ref{t_cube}, pending the proof of \eqref{G-int}. 
\begin{lemma}[\eqref{G-int}]
For each $a > 0$,
\[
G(a) = \int_{-a/2}^{a/2} g(x)\, dx = \dfrac{4}{\pi} \sum_{k \geq 0} 
\frac{(-1)^k}{1+2 k} \exp\left( -\frac{(1+2k)^2\pi^2}{2a^2}\right).
\]
\end{lemma}
\begin{proof}
This integral is computed 
using the Poisson summation 
formula:
\begin{equation}\label{poisson}
\sum_{k = -\infty}^{\infty} f(k) = 
\sum_{k = -\infty}^{\infty} \int_{\rr} f(y) e^{-2\pi i k y} \, dy,
\end{equation}
where \( f \) satisfies certain conditions (e.g., Schwartz). See \S\,6 in \cite{Bellman}. 
\smallskip%

We write:
\[
g(x) = (2\pi)^{-1/2} \sum_{k} (-1)^k 
e^{-(x+k a)^2/2}\, dx = 
(2\pi)^{-1/2} e^{-x^2/2} 
\sum_k 
e^{k (\pi i - a x) - k^2 a^2/2},
\]
and using \eqref{poisson}, 
\[
(2\pi)^{-1/2} e^{-x^2/2} 
\sum_k 
e^{k (\pi i - a x) - k^2 a^2/2} = 
(2\pi)^{-1/2} e^{-x^2/2} \sum_k \int_{\rr} e^{y (\pi i - a x) - y^2 a^2/2} \cdot 
e^{-2 \pi i k y}\, dy,
\]
we find:
\[
(2\pi)^{-1/2} \sum_{k} (-1)^k 
e^{-(x+k a)^2/2}\, dx = 
a^{-1} e^{-\pi^2/2a^2} 
\sum_k e^{(2k-1)\pi i x/a} 
\cdot e^{2 \pi^2 k/a^2 - 2\pi^2 k^2/a^2}.
\]
Calculating
\[
\int_{-a/2}^{a/2} e^{(2k-1)\pi i x/a}\, dx = 
(-1)^k \dfrac{2 a}{(1-2k)\pi},
\]
we get:
\begin{multline*}
\int_{-a/2}^{a/2} g(x)\, dx = 
\dfrac{2}{\pi} e^{-\pi^2/2 a^2} 
\sum_k \dfrac{(-1)^k}{1+2k} e^{-2 \pi^2 k/a^2 - 2\pi^2 k^2/a^2} = \\
= \dfrac{2}{\pi}  
\sum_k \dfrac{(-1)^k}{1+2k} 
e^{-\pi^2/2a^2 \cdot (1+2k)^2} = 
\dfrac{4}{\pi}  
\sum_{k \geq 0} \dfrac{(-1)^k}{1+2k} 
e^{-\pi^2/2a^2 \cdot (1+2k)^2},
\end{multline*}
and the lemma follows. \qed
\end{proof}

\section{Proof of (2) of Propositions \ref{t_main} and \ref{t_cube}}\label{generator}
To design the adversary's strategy, we use the initial sample \( X_1, \ldots, X_n \) and the value \( a \) to generate a new normal sample \( X'_1, \ldots, X'_n \), where for each \( i \), \( X'_i - X_i \in \{-a, 0, a\} \). To begin with, we construct a discrete-time Markov chain on \( \mathbb{R} \) with the following properties. For \( x, y \in \mathbb{R} \), let \( q(y|x) \) denote the probability of transitioning to \( y \) from \( x \). The transition probabilities are as follows:
\[
q(x + a \mid x) = \varphi(x), \quad q(x \mid x) = \gamma(x), \quad q(x - a \mid x) = 1 - \varphi(x) - \gamma(x),
\]
and \( q(y \mid x) = 0 \) for other values of \( y \).
\smallskip%

Let \( p \) be the density of the standard normal distribution \( N(0,1) \), 
defined as:
\[
p(x) = (2\pi)^{-1/2} e^{-x^2/2}.
\]
The functions \( \varphi \) and \( \gamma \) are chosen to satisfy the following condition:
\begin{equation}\label{stationary}
\varphi(x - a) p(x - a) + \gamma(x) p(x) + \left[1 - \varphi(x + a) - \gamma(x + a)\right] p(x + a) = p(x),
\end{equation}
which ensures that the standard normal distribution is the stationary distribution of the Markov chain. 
\smallskip%

In addition, the function $\gamma$ is designed to satisfy:
\begin{equation}\label{int_gamma}
\int \gamma(x) p(x)\, dx = G(a),
\end{equation}
where $G(a)$ is defined by \eqref{G}. In effect, this condition bounds the probability that the Markov chain remains in place after one step. 
\smallskip%

Before constructing the functions \( \varphi \) and \( \gamma \), first show how they will be used in the proof. The adversary's strategy for constructing \( X'_1, \ldots, X'_n \) proceeds as follows: Given \( X_1, \ldots, X_n \), the adversary starts with \( X_1 \) as the initial state of the Markov chain defined above. After one transition step, \( X_1 \) moves to a new state, denoted \( Y_1 \). Since \( X_1 \) is Gaussian, \( Y_1 \) is also Gaussian. The adversary repeats this process for each observation, obtaining an iid sequence \( Y_1, \ldots, Y_n \), each 
\( Y_i \) being \( N(0,1) \).
\smallskip%

Let \( T \) be the event:
\[
T = \left\{ \frac{\#\{ i : X_i = Y_i \}}{n} < G(a) + \varepsilon \right\}.
\]
If \( T \) occurs, set \( X'_i = Y_i \) for each \( i \). Otherwise, choose \( X'_i \) arbitrarily as either \( X_i - a \) or \( X_i + a \). Conditioning on \( T \), the adversary can challenge any test \( \Omega \) using \( X'_1, \ldots, X'_n \), achieving the following lower bound on the false-negative rate:
\begin{multline*}
P\left( \text{\normalfont{accept }} H_0 \,|\, H_1 \right) = 
P((X'_1, \ldots, X'_n) \in \Omega) \geq 
P(T \cap (Y_1, \ldots, Y_n) \in \Omega) \geq \\
 \geq P(T) + P((Y_1, \ldots, Y_n) \in \Omega) - 1 = 
P(T) + P((X_1, \ldots, X_n) \in \Omega) - 1.
\end{multline*}
Since $P\left( \text{\normalfont{accept }} H_1 \,|\, H_0 \right) = 1 - P((X_1, \ldots, X_n)$, it follows that
\[
P\left( \text{\normalfont{accept }} H_0 \,|\, H_1 \right) + 
P\left( \text{\normalfont{accept }} H_1 \,|\, H_0 \right) \geq P(T).
\]
It remains to show that \( T \) occurs with high probability. Define \( Z_i \) as:
\[
Z_i = 
\begin{cases} 
1 & \text{if } X_i = Y_i, \\
0 & \text{otherwise}.
\end{cases}
\]
Since $P(Z_i = 1\,|\,X_i) = \gamma(X_i)$, it follows from \eqref{int_gamma} that
\[
P(Z_i = 1) = 
\int \gamma(x) p(x)\, dx = G(a).
\]
Letting $S_n = Z_1 + \cdots + Z_n$, Hoeffding's bound gives:
\[
P(\overline{T}) = 
P(n^{-1} S_n - G(a) \geq \varepsilon) \leq e^{-2 n \varepsilon^2},
\]
completing the proof of Proposition \ref{t_main}.
\smallskip%

Set $a = c/\sqrt{\ln{n}}$. To establish Proposition \ref{t_cube}, it suffices to show 
that if $c < \pi/\sqrt{2}$, then $S_n = 0$ with high probability. Indeed:
\begin{multline}\label{cube_O}
P(S_n = 0) = (1 - G(a))^n 
\stackrel{\eqref{G-int}}{=}  
\left(
1 - 
\dfrac{4}{\pi} \sum_{k \geq 0} 
\frac{(-1)^k}{1+2 k} 
e^{-(1+2k)^2 \cdot \pi^2/2a^2 }
\right)^n > 
\left(
1 - 
\dfrac{4}{\pi} e^{-\pi^2/2a^2}
\right)^n = \\
= \left(
1 - 
\dfrac{4}{\pi} n^{-\pi^2/2c^2}
\right)^n = 
1 - O\left( n^{-\pi^2/2c^2} \right).
\end{multline}
The rest of this section shows the existence of functions \( \gamma \) and \( \varphi \) that satisfy the necessary conditions. To begin with, define $\gamma$ as:
\[
\gamma(x) =
\begin{cases} 
    \sum_{k = -\infty}^{\infty} (-1)^k e^{-k a x - k^2 a^2/2},  & \text{if } x \in (-a/2, a/2), \\
    0, & \text{otherwise}.
\end{cases}
\]
For \( x \) where \( \gamma(x) \neq 0 \), it relates to the function \( g \) defined by \eqref{g} via
\[
p(x) \gamma(x) = g(x).
\]
Thus, to establish \eqref{int_gamma}, we have
\[
\int_{\mathbb{R}} \gamma(x) p(x) \, dx = 
\int_{-a/2}^{a/2} g(x) \, dx = G(a).
\]
Since \( \gamma \) represents a transition probability, it must satisfy
\[
0 \leq \gamma(x) \leq 1 \quad 
\text{for each } x.
\]
Let us first prove that \( \gamma \) is non-negative, i.e.,
\begin{equation}\label{positive}
\gamma(x) = \sum_{k} (-1)^k e^{-k a x - k^2 a^2/2} > 0 \quad \text{for all } |x| < a/2.
\end{equation}
Setting $b = -a^2$ and $z = \pi i - x a$, we can rewrite the series as:
\[
\sum_k (-1)^k e^{-k a x - k^2 a^2/2} = 
\sum_k e^{\pi i k - k a x - k^2 a^2/2} = 
\sum_k e^{k z + b k^2/2}.
\]
This series, representing the Jacobi theta 
function $\vartheta_3(z)$, has zeros forming a lattice with basis $2 \pi i$, $b$, centered at $z_0 = \pi i + b/2$. 
This classical result is found in \S\,2 in \cite{Dub}. See also \cite{Bellman}. In the $x$-coordinate, the zeros are given by 
\[
\dfrac{a}{2} + \dfrac{2 \pi i\, n}{a} + a\, m\quad \text{for all $n,m \in \zz$}.
\]
Consequently, 
\begin{equation}\label{nozero}
\sum_k (-1)^k e^{-k a x - k^2 a^2/2} \neq 0\quad \text{for each $a > 0$ 
and $|x| < a/2$,}
\end{equation}
and this series vanishes at $x = \pm a/2$. Setting $x = 0$, we find:
\[
\sum_k (-1)^k e^{-k^2 a^2/2} \to 1\quad 
\text{as $a \to \infty$.}
\]
confirming \eqref{positive} for $x = 0$ and large $a$. But then, using \eqref{nozero}, we extend \eqref{positive} to all $a > 0$ and $|x| < a/2$. Thus, we have shown that \( \gamma \) is non-negative.
\smallskip%

Let us now show that \( \gamma(x) \leq 1 \) for each \( x \). Assuming \( |x| < a/2 \), we find:
\[
\gamma(x) p(x) = \sum_k (-1)^k p(x + k a) = p(x) + 
\sum_{k \geq 1} (-1)^{k} p(x + k a) + \sum_{k \leq -1} (-1)^{k} p(x + k a) < p(x),
\]
since \( p \) is an even function and monotonically decreases for \( x > 0 \). 
\smallskip%

With \( \gamma \) constructed, we now find \( \varphi \). To do this, we replace the stationarity condition \eqref{stationary} with the stronger \emph{detailed balance} condition:
\[
\varphi(x) p(x) = \left[ 1 - \gamma(x+a) - \varphi(x+a) \right] p(x+a).
\]
A Markov chain satisfying detailed balance for \( p \) necessarily has \( p \) as a stationary distribution. At the same time, this equation is straightforward to solve:
\[
p(x) \varphi(x) = \sum_{k \geq 1} (-1)^{k+1} \left[ 1 - \gamma(x + k a) \right] p(x + k a).
\]
This formula recovers \( \varphi \). It remains to verify that \( 0 \leq \varphi(x) \leq \varphi(x) + \gamma(x) \leq 1 \) for all \( x \). This is indeed true, and the verification follows similarly to the argument for \( 0 \leq \gamma(x) \leq 1 \) above. We omit the detailed calculation for brevity.

\bibliographystyle{plain}
\bibliography{ref}

\end{document}